\documentclass[a4paper,12pt]{amsart}
\usepackage{amsmath, amssymb, amsthm}
\usepackage{verbatim}
\usepackage{hyperref}

\newcounter{theorem}
\newtheorem{thm}[theorem]{Theorem}
\newtheorem{lemma}[theorem]{Lemma}
\newtheorem{prop}[theorem]{Proposition}
\newtheorem{cor}[theorem]{Corollary}
\newtheorem{defn}[theorem]{Definition}

\theoremstyle{remark}
\newtheorem*{remark*}{Remark}
\newtheorem{remark}[theorem]{Remark}

\numberwithin{equation}{section}
\numberwithin{theorem}{section}

\newcommand{\e}{\epsilon}

\newcommand{\C}{\mathbb{C}}

\renewcommand{\setminus}{\backslash}
\newcommand{\K}{\mathcal{K}}
\newcommand{\tens}{\otimes}
\newcommand{\dsum}{\oplus}
\newcommand{\bigdsum}{\bigoplus}

\newcommand{\iso}{\cong}
\newcommand{\dunion}{\amalg}
\newcommand{\bigdunion}{\coprod}

\newcommand{\Cu}{\mathcal{C}u}
\newcommand{\catCu}{\mathbf{Cu}}
\newcommand{\jsZ}{\mathcal{Z}}

\newcommand{\diag}{\mathrm{diag}}
\newcommand{\ev}{\mathrm{ev}}
\newcommand{\rank}{\mathrm{Rank}\,}
\newcommand{\Prim}{\mathrm{Prim}}
\newcommand{\id}{\mathrm{id}}

\newcommand{\labelledthing}[2]{\hspace{4pt}\buildrel {#2} \over #1 \hspace{3pt}} 

\newcommand{\labelledrightarrow}{\labelledthing{\longrightarrow}}

\newcommand{\ccite}[2]{\cite[#1]{#2}}
\newcommand{\alabel}{\label}

\newcommand{\X}{\mathbb{X}}
\newcommand{\Y}{\mathbb{Y}}
\newcommand{\rdratio}{\mathrm{R_{dim:rank}}}
\newcommand{\rc}{\mathrm{rc}}

\begin{document}
\title{Regularity for stably projectionless, simple $C^*$-algebras}
\author{Aaron Tikuisis}


\begin{abstract}
This paper explores the following regularity properties and their relationships for simple, not-necessarily-unital $C^*$-algebras: 
\begin{enumerate}
\item Jiang-Su stability,
\item unperforation in the Cuntz semigroup, and 
\item slow dimension growth (applying only in the case that the $C^*$-algebra is approximately subhomogeneous).
\end{enumerate}
An example is given of a simple, separable, nuclear, stably projectionless $C^*$-algebra whose Cuntz semigroup is not almost unperforated.
This example is in fact approximately subhomogeneous.
It is also shown that, in contrast to this example, when an approximately subhomogeneous simple $C^*$-algebra has slow dimension growth, its Cuntz semigroup is necessarily almost unperforated.
\end{abstract}

\maketitle
\section{Introduction}
The (Elliott) classification program for $C^*$-algebras consists of the goal of classifying simple, separable, nuclear $C^*$-algebras by $K$-theory and traces, as outlined by George Elliott in 1990. 
The conjecture that all simple, separable, nuclear $C^*$-algebras could be thus classified was, however, disproven by Andrew Toms in \cite{Toms:Independence}, who provided a unital counterexample (see also the more refined counterexample in \cite{Toms:annals}).
The reaction has been to try to characterize which simple, separable, nuclear $C^*$-algebras are sufficiently ``well-behaved'' to expect such a classification.
This has resulted in a conjecture that different regularity properties, i.e. notions of being ``well-behaved,'' are equivalent, and moreover, that the class of $C^*$-algebras which satisfy these (hopefully equivalent) regularity properties can be classified (see \ccite{Remark 3.5}{TomsWinter:V1} and also the expository article \cite{ElliottToms}).
The different regularity properties considered include notions of low topological dimension, stability under tensoring with the Jiang-Su algebra $\jsZ$, and unperforation of the Cuntz semigroup.
(In particular, perforation in the Cuntz semigroup of a $C^*$-algebra $A$ in \cite{Toms:annals} is what distinguishes it from $A \tens \jsZ$, although the $K$-theory and traces of these two algebras agree, along with many other classical invariants.)

Although certainly incomplete, the problems of reconciling the different notions of a regularity and classifying the well-behaved $C^*$-algebras have seen remarkable recent progress towards solutions (\cite{Lin:AsympClassification,Winter:pure} in particular) -- in the unital case.
Little attention, however, has been paid to the nonunital case.
Note that if a simple $C^*$-algebra $A$ is such that its stabilization contains a nonzero projection $p$ then $A$ is Morita equivalent to the unital $C^*$-algebra $p(A \tens \K)p$ by Brown's Theorem \cite{Brown:StableIsomorphism}; thus we shall talk about stably projectionless $C^*$-algebras instead of merely nonunital ones.

This paper explores the notions of regularity for simple but not necessarily unital $C^*$-algebras.
The most important result, perhaps, is Theorem \ref{MainThm}, giving an example of a stably projectionless $C^*$-algebra with perforation in its Cuntz semigroup.
This $C^*$-algebra is approximately subhomogeneous, and its construction draws heavily on techniques of Villadsen \cite{Villadsen:Perforation}, along with their refinement in \cite{Toms:annals}.
On the other hand, it is shown in Corollary \ref{SDGunperf} that an approximately subhomogeneous $C^*$-algebra constructed by a system with slow dimension growth cannot have perforation in its Cuntz semigroup; this result was achieved in the unital case in \ccite{Theorem 1.1}{Toms:comparison}, and this proof adapts the same techniques, and in particular, makes use of an adaptation of Toms' radius of comparison to the nonunital situation.

The author and Toms show in \cite{TikuisisToms:InProgress} that the Cuntz semigroups of approximately subhomogeneous $C^*$-algebras with slow dimension growth are almost divisible.
In conjunction with Corollary \ref{SDGunperf}, this yields a complete computation of these Cuntz semigroups.

The neglect of stably projectionless $C^*$-algebras in the literature likely stems in part from the fact that many important examples (including simple approximately homogeneous $C^*$-algebras and crossed products of unital $C^*$-algebras by discrete groups) are excluded from this class.
There are, however, interesting known examples of stably projectionless $C^*$-algebras, including many crossed products of Cuntz algebras by trace-scaling automorphisms \cite{KishimotoKumjian:projless} and one algebra in particular which resembles in many ways a finite version of $\mathcal{O}_2$ \cite{Jacelon:SSA}.

A further reason that research has been focused on the unital case is simply that it admits many simplifications; it is unclear whether these simplifications just make the proofs easier or represent genuine differences in the structure of the algebras involved.
The problems tackled here arise from considering exactly this question, and all the results of this paper can be summarized by saying that nonunital simple $C^*$-algebras behave analogously to unital ones, although the proofs are more involved.
Much of these additional technicalities are contained in theory developed regarding nonunital (approximately) subhomogeneous $C^*$-algebras in Section \ref{RSH-section} (which overlaps largely with Chapter 3 of the author's Ph.D.\ thesis \cite{MyThesis}).
This includes a nonunital generalization (Definition \ref{RSH-defn}) of Phillips' notion of recursive subhomogeneous algebras, and structural results (Corollary \ref{ASH-Structure-Nonunital} and Theorem \ref{ASH-SimpleStructure}) showing that separable approximately subhomogeneous algebras can be expressed as certain nice limits of recursive subhomogeneous algebras.

Some other contributions to our understanding of stably projectionless $C^*$-algebras can be found in \cite{Jacelon:SSA}, where a particular self-absorbing example is constructed with potentially important properties, and in \cite{Robert:NCCW}, where a broad class of not necessarily unital $C^*$-algebras is classified.

Section \ref{Prelim-section} contains preliminary material regarding the Cuntz semigroup and approximately subhomogeneous algebras.
Theory regarding non-unital recursive subhomogeneous $C^*$-algebras can be found in Section \ref{RSH-section}.
Section \ref{Perforation-section} contains the example of a simple, stably projectionless, approximately subhomogeneous $C^*$-algebra with perforated Cuntz semigroup (to be precise, the Cuntz semigroup is not almost unperforated).
Finally, Section \ref{SDG-section} introduces the notion of slow dimension growth for nonunital $C^*$-algebras and shows that it implies that the Cuntz semigroup is almost unperforated.

\section*{Acknowledgements}
The author would like to thank the Fields Institute and the Sonderforschungsbereich 878 for their hospitality during the time that this research was done.
Discussions with George Elliott, Zhiqiang Li, Henning Petzka, and Wilhelm Winter contributed to the research presented here.

\section{Preliminaries}
\alabel{Prelim-section}

Although the chief results of this paper concern simple $C^*$-algebras, ideals and quotients of $C^*$-algebras will frequently appear in the underlying theory and proofs.
The symbol $\pi_I$ will always be used to denote the quotient map $A \to A/I$ when $I$ is an ideal of a $C^*$-algebra $A$.

\subsection{The Cuntz semigroup}
\begin{defn}
Let $A$ be a $C^*$-algebra and let $a,b \in (A \tens \K)_+$.
We say that $a$ is \textbf{Cuntz below} $b$ if there exists a sequence $(s_n) \in A \tens \K$ such that
\[ \|a - s_nbs_n^*\| \to 0 \]
as $n \to \infty$.
We say that $a$ is \textbf{Cuntz equivalent} to $b$ if each of $a$ and $b$ is Cuntz below the other.

We use the notation $[a]$ to denote the Cuntz-equivalence class of the element $a$, and write
\[ [a] \leq [b] \]
if $a$ is Cuntz below $b$ (this is well-defined, since the Cuntz below relation is transitive).
\end{defn}

This version of the Cuntz semigroup was introduced in \cite{CowardElliottIvanescu}; the original definition used $\bigcup_n A \tens M_n$ in place of $A \tens \K$.
The version using $A \tens \K$ has a rich structure, largely described by belonging to the category $\catCu$ as to be defined presently.
In order to define objects of this category, we need the notion of (order-theoretic, sequential) compact containment.
Let $T$ be a preordered set with $x,y \in T$.
We say that $x$ is \textbf{compactly contained} in $y$---denoted by $x \ll y$---if for any increasing sequence $(z_n)$ in $T$ with supremum $z \geq y$, there exists $n_0$ such that $x \leq y_{n_0}$.

An object $S$ of $\catCu$ is an ordered semigroup with the following properties.
\begin{enumerate}
\item[{\bf P1}] $S$ contains a zero element; \vspace{1mm}
\item[{\bf P2}] the order on $S$ is compatible with addition:
\[ x_1 + x_2 \leq y_1 + y_2 \]
whenever $x_i \leq y_i$ for $i = 1,2$;\vspace{1mm}
\item[{\bf P3}] every increasing sequence in $S$ has a supremum; \vspace{1mm}
\item[{\bf P4}] for every $x \in S$, there exists a sequence $(x_n) \subset S$ such that
$x_n \ll x_{n+1}$ for every $n$ and $\sup_n x_n = x$; \vspace{1mm}
\item[{\bf P5}] the operation of passing to the supremum of an increasing sequence and the relation $\ll$ are compatible with addition:  if $(s_n)$ and $(t_n)$ are increasing sequences then
\[ \sup (s_n + t_n) = (\sup s_n) + (\sup t_n); \]
and if $x_i \ll y_i$ for $i =1,2$, then
\[ x_1 + x_2 \ll y_1 + y_2. \]
\end{enumerate}

\noindent
Here we assume further that $0 \leq x$ for any $x \in S$.  This is always the case for the Cuntz semigroup of a $C^*$-algebra.
For objects $S$ and $T$ of $\catCu$, the map $\phi\colon S\to T$ is a morphism in the category $\catCu$ if

\begin{enumerate}
\item[{\bf M1}] $\phi$ is order preserving; \vspace{1mm}
\item[{\bf M2}] $\phi$ is additive and maps $0$ to $0$; \vspace{1mm}
\item[{\bf M3}] $\phi$ preserves the suprema of increasing sequences; \vspace{1mm}
\item[{\bf M4}] $\phi$ preserves the relation $\ll$. \vspace{1mm}
\end{enumerate}

The category $\catCu$ admits inductive limits, and $\Cu(\cdot)$ may be viewed as a functor from $C^*$-algebras into $\catCu$. A central result of \cite{CowardElliottIvanescu} is that if $(A_i,\phi_i)$ is an inductive sequence of $C^*$-algebras, then
\[
\Cu \left( \lim_{i \to \infty} (A_i,\phi_i) \right) \cong \lim_{i \to \infty} (\Cu(A_i), \Cu(\phi_i)).
\]
Let $S = \lim_{i \to \infty}(S_i,\phi_i^{i+1})$ be an inductive limit in the category $\catCu$, with $\phi_i^j\colon S_i \to S_j$ and
$\phi_i^\infty\colon S_i \to S$ the canonical maps.  We have the following two properties (established in the proof of \ccite{Theorem 2}{CowardElliottIvanescu}):
\begin{enumerate}
\item[{\bf L1}] each $x \in S$ is the supremum of a sequence of the form $(\phi_i^\infty(x_i))$ where $x_i \in S_i$ and $\phi_i^{i+1}(x_i) \ll x_{i+1}$ for all $i$; \vspace{1mm}
\item[{\bf L2}] If $x,y\in S_i,$ and $\phi_i^\infty(x)\leq \phi_i^\infty(y)$, then for all $x'\ll x$ there is $j \geq i$ such that
\[ \phi_i^j(x')\leq \phi_i^j(y). \]
\end{enumerate}

\subsection{Approximately subhomogeneous algebras}
Approximately subhomogeneous $C^*$-algebras appear often in the literature; see \cite{Elliott:ashrange,NgWinter:ash,Phillips:arsh,Toms:rigidity,Winter:pure} which includes many strong structural results, most of which concern only the unital case.
Here we shall be chiefly concerned with nonunital (and in fact, stably projectionless) approximately subhomogeneous algebras.

\begin{defn}
A $C^*$-algebra is \textbf{subhomogeneous} if there is a finite bound on the dimensions of its irreducible representations.
A $C^*$-algebra is said to be \textbf{approximately subhomogeneous} if it can be written as an inductive limit of a sequence of subhomogeneous algebras.
\end{defn}

\section{Inductive limit structure of stably projectionless approximately subhomogeneous algebras}
\alabel{RSH-section}

\subsection{Recursive subhomogeneous algebras and their ideals}
The systematic study of unital approximately subhomogeneous algebras has been largely aided by the computational device of recursive subhomogeneous algebras, which are certain subhomogeneous algebras with particularly accessible structure.
Every unital approximately subhomogeneous algebra has been shown to be an inductive limit of recursive subhomogeneous algebras, which is why we are able to use them to study approximately subhomogeneous algebras.
Here, we give a natural expansion of the class of recursive subhomogeneous algebras to include nonunital algebras, and show in Corollary \ref{ASH-Structure-Nonunital} that, with this expansion, the aforementioned result continues to hold in the nonunital case.
The definition appears first in the author's Ph.D.\ thesis \cite{MyThesis}, as does Corollary \ref{ASH-Structure-Nonunital}.

\begin{defn}
\alabel{RSH-defn}
The class of recursive subhomogeneous algebras is the smallest class $\mathcal{RSH}$ containing $0$ and closed under a certain pullback construction as follows.
If $R \in \mathcal{RSH}$, $\Omega$ is a compact Hausdorff space, $\Omega^{(0)}$ is a closed (possibly empty) subset of $\Omega$, $n \in \{1,2,\dots\}$, $\rho:R \to C(\Omega^{(0)},M_n)$ is a $*$-homomorphism, and
\[
\begin{array}{rcl}
R' & \to & C(\Omega, M_{n}) \\
\downarrow && \quad\ \downarrow {\scriptstyle f \mapsto f|_{\Omega^{(0)}}} \\
R & \labelledrightarrow{\rho} & C(\Omega^{(0)}, M_{n});
\end{array}
\]
is a pull-back then $R' \in \mathcal{RSH}$.
Explicitly, we may identify the pullback $R'$ with the amalgamated direct sum,
\[ \{(f,a) \in C(\Omega, M_n) \dsum R: f|_{\Omega^{(0)}} = \rho(a)\}. \]
\end{defn}

\begin{remark*}
\begin{enumerate}
\item
In \cite{Phillips:rsh}, Chris Phillips originally defined recursive subhomogeneous algebras just as above, except that the maps $\rho$ are required to be unital.
Phillips' class of recursive subhomogeneous algebras consists exactly of the unital $C^*$-algebras which are recursive subhomogeneous as defined here (the proof of this is quite straightforward).

\item
A further ``non-unital'' relaxation of the definition would be to drop the condition that $\Omega$ is compact, and instead require that it be only locally compact.
However, making use of the one-point compactification, one sees that the class of recursive subhomogeneous algebras is unchanged: the pull-back $R'$ in
\[
\begin{array}{rcl}
R' & \to & C_0(\Omega, M_{n}) \\
\downarrow && \quad\ \downarrow {\scriptstyle f \mapsto f|_{\Omega^{(0)}}} \\
R & \labelledrightarrow{\rho} & C_0(\Omega^{(0)}, M_{n})
\end{array}
\]
is isomorphic to the pull-back $R''$ in
\[
\begin{array}{rcl}
R'' & \to & C(\Omega \cup \{\infty\}, M_{n}) \\
\downarrow && \quad\ \downarrow {\scriptstyle f \mapsto f|_{\Omega^{(0)} \cup \{\infty\}}} \\
R & \labelledrightarrow{\tilde\rho} & C(\Omega^{(0)} \cup \{\infty\}, M_{n}),
\end{array}
\]
where $\tilde\rho(a)|_{\Omega^{(0)}} = \rho(a)$ and $\tilde\rho(a)(\infty) = 0$.

\item
Evidently, every recursive subhomogeneous algebra can be realized as an iterated pullback.
That is to say, if $R$ is a recursive subhomogeneous algebra then there exist algebras $R_0,\dots,R_\ell$ such that  $R_0 = 0$, $R_\ell = R$, and for each $i=1,\dots,\ell$, $R_i$ is given as a pull-back
\[
\begin{array}{rcl}
R_i & \labelledrightarrow{\sigma_i} & C(\Omega_i, M_{n_i}) \\
{\scriptstyle \lambda_i^{i-1}} \downarrow && \quad\ \downarrow {\scriptstyle f \mapsto f|_{\Omega_i^{(0)}}} \\
R_{i-1} & \labelledrightarrow{\rho_i} & C(\Omega_i^{(0)}, M_{n_i});
\end{array}
\]
where $\Omega_i$ is a compact Hausdorff space, $\Omega_i^{(0)}$ is a closed subset, and $\rho_i$ is a $*$-homomorphism.
Following \ccite{Definition 1.2}{Phillips:rsh}, we call this a \textbf{(recursive subhomogeneous) decomposition} of $R$.
We call $\Omega := \bigdunion_i \Omega_i$ the \textbf{total space}, and $\max_i n_i$ the \textbf{maximum matrix size}, of this decomposition.
The \textbf{topological dimension} of the decomposition refers to the covering dimension of the total space.

The recursive subhomogeneous decomposition gives rise to a canonical representation
\[ \sigma:R \to C(\Omega, \K); \]
namely, it is defined by 
\[ \sigma(\cdot)|_{\Omega_i} = \sigma_i \circ \lambda_{i+1}^i \circ \cdots \circ \lambda_\ell^{\ell-1}, \]
for each $i=1,\dots,\ell$.

\end{enumerate}
\end{remark*}

\begin{prop}(cf.\ \ccite{Corollary 3.3}{Phillips:rsh}) \alabel{RSH-ideal}
An ideal of a recursive subhomogeneous algebra is itself a recursive subhomogeneous algebra.
\end{prop}

\begin{proof}
Let $R$ be a recursive subhomogeneous algebra and let $J$ be an ideal of $R$.
As the class of recursive subhomogeneous algebras has a recursive definition, we may use induction for this proof.
Specifically, we first observe that the result holds in the base case that $R=M_n$.
Then, for the inductive step, we suppose that $R$ be given by the pullback
\[
\begin{array}{rcl}
R & \labelledrightarrow{\sigma} & C(\Omega, M_{n}) \\
{\scriptstyle \lambda} \downarrow && \quad\ \downarrow {\scriptstyle f \mapsto f|_{\Omega^{(0)}}} \\
R_0 & \labelledrightarrow{\rho} & C(\Omega^{(0)}, M_{n});
\end{array}
\alabel{RSH-ideal-pullback}
\]
where our inductive hypothesis says that any ideal of $R_0$ is a recursive subhomogeneous algebra.

Since the restriction map is surjective, so is the map $\lambda$.
From this, it follows that $J_0:=\lambda(J)$ is an ideal of $R_0$, and therefore it is itself a recursive subhomogeneous algebra.
Define $\Lambda$ to be the open subset of $\Omega$ such that $C_0(\Lambda,M_n)$ is the ideal generated by $\sigma(J)$.
By commutativity of \eqref{RSH-ideal-pullback}, $\rho(J_0)$ is contained in $C_0(\Lambda \cap \Omega,M_n)$, and thus,
\begin{equation}\alabel{RSH-ideal-newpullback}
\begin{array}{rcl}
J & \labelledrightarrow{\sigma|_{J}} & C_0(\Lambda, M_{n}) \\
{\scriptstyle \lambda|_{J}} \downarrow && \quad\ \downarrow {\scriptstyle f \mapsto f|_{\Omega^{(0)} \cap \Lambda}} \\
J_{0} & \labelledrightarrow{\rho|_{J_{0}}} & C_0(\Lambda \cap \Omega^{(0)}, M_{n})
\end{array}
\end{equation}
commutes.
We shall show that, in fact, \eqref{RSH-ideal-newpullback} is a pullback, from which it follows (by the remark (ii) after Definition \ref{RSH-defn}) that $J$ is recursive subhomogeneous.
To show that \eqref{RSH-ideal-newpullback} is a pullback, we need only show that for any $a \in R$, if $\sigma(a) \in C_0(\Lambda,M_{n})$ and $\lambda(a) \in J_0$ then $a \in J$.

Let $\e > 0$.
Since 
\[ \sigma(a)|_{\Omega^{(0)}} \in \rho(\lambda(J)) = \sigma(J)|_{\Omega^{(0)}}, \]
there must be some open set $U$ containing $\Omega^{(0)}$ such that $\sigma(a)|_{\overline{U}}$ is approximately contained (to within $\e$) in $\sigma(J)|_{\overline{U}}$.
On the other hand, since $U$ contains $\Omega^{(0)}$, the map $a \mapsto \sigma(a)|_{\Omega \setminus U}$ is surjective, so that $\sigma(J)|_{\Omega \setminus U}$ is an ideal of $C(\Omega \setminus U, M_n)$ and therefore
\[ \sigma(J)|_{\Omega \setminus U} = C_0(\Lambda \setminus U, M_n). \]
In particular, we see that $\sigma(a)|_{\Omega \setminus U} \in  \sigma(J)|_{\Omega \setminus U}$.

If $(e_\alpha)$ is an approximate identity for $J$ then we see that
\[ \limsup \|\sigma(e_\alpha a) - \sigma(a)|_{\overline{U}}\| \leq \e \]
while
\[ \|\sigma(e_\alpha a) - \sigma(a)|_{\Omega \setminus U}\| \to 0 \]
and $\|\lambda(e_\alpha a) - \lambda(a)\| \to 0$.
Hence, there exists $\alpha$ such that $\|a - e_\alpha a\| \leq 2\epsilon$, and of course, $e_\alpha a \in J$, which proves that $a$ has distance at most $2\e$ from $J$.
Since $\e$ is arbitrary, it follows that $a \in J$.
\end{proof}

In \ccite{Theorem 2.16}{Phillips:rsh}, unital recursive subhomogeneous algebras are characterized abstractly (in the separable case). 
A generalization to the non-unital case follows as a consequence of the last proposition.
We shall denote by $\Prim(A)$ the primitive spectrum of a $C^*$-algebra $A$, with the kernel-hull topology (for details, see for instance \ccite{Chapter 4}{Pedersen:CstarBook}).
We use $\Prim_n(A)$ to denote the subset of $\Prim(A)$ consisting of the kernels of irreducible representations of dimension exactly $n$.

\begin{cor}\alabel{RSH-Characterization}
Let $A$ be a separable unital $C^*$-algebra, and let $N,d$ be natural numbers.
The following are equivalent.
\begin{enumerate}
\item $A$ has a recursive subhomogeneous decomposition with maximum matrix size at most $N$ and topological dimension at most $d$;
\item $A$ has a recursive subhomogeneous decomposition with maximum matrix size at most $N$ and topological dimension at most $d$, and whose total space is at most second countable;
\item All irreducible representations of $A$ have dimension at most $N$, and for $n=1,\dots,N$, the covering dimension of $\Prim_n(A)$ is at most $d$.
\end{enumerate}
In particular, any finitely generated subhomogeneous $C^*$-algebra is recursive subhomogeneous with finite-dimensional total space.
\end{cor}

\begin{proof}
In the unital case, this is exactly \ccite{Theorem 2.16}{Phillips:rsh}.
In the nonunital case, while (ii) $\Rightarrow$ (i) is immediate and (i) $\Rightarrow$ (iii) is quite straightforward (see the proof of \ccite{Phillips 2.16}{Phillips:rsh}), let us explain how to get (iii) $\Rightarrow$ (ii).
Assuming (iii), we see that (iii) also holds for the unitization $A^\sim$ of $A$, since
\[
\Prim_n(A^\sim) = \begin{cases} \Prim_1(A) \dunion \{\cdot\},\ \ &\text{if } n=1, \\ \Prim_n(A),\ \ &\text{otherwise}. \end{cases}
\]
Thus, by the unital case, (ii) holds for $A^\sim$.
By Proposition \ref{RSH-ideal} (and its proof), (ii) holds also for $A$, as required.

The last statement follows from the proof of \ccite{Theorem 1.5}{NgWinter:ash}, where it is shown that if $A$ is a subhomogeneous $C^*$-algebra generated by $m$ elements then
\[ \dim \Prim_n(A) \leq 4mn^2, \]
for all $n$.
\end{proof}

\begin{cor}\alabel{ASH-Structure-Nonunital} (cf.\ \ccite{Corollary 2.1}{NgWinter:ash})
Every separable approximately subhomogeneous algebra can be written as an inductive limit of recursive subhomogeneous algebras with finite-dimensional total space.
\end{cor}

\begin{proof}
Let $A$ be a separable approximately subhomogeneous algebra.
Since $A$ is separable and an inductive limit of subhomogeneous algebras, it is easy to see that it is an inductive limit of finitely generated subhomogeneous algebras.
By Corollary \ref{RSH-Characterization}, it follows that these finite stage algebras are recursive subhomogeneous with finite-dimensional total space.
\end{proof}

\subsection{Compact primitive spectrum for finite stages of simple approximately subhomogeneous algebras}
Here we improve on Corollary \ref{ASH-Structure-Nonunital} in the simple case, by showing that in this case, the recursive subhomogeneous algebras in the inductive limit can be chosen such that their primitive ideal space is compact.
This fact is often used --- implicitly --- in the unital case, where it is obvious: a unital inductive limit must come from a unital inductive system, and every unital algebra has compact primitive ideal space.
In the nonunital case, the key fact used is that when we tensor a simple algebra with a stable purely infinite simple algebra (such as $\mathcal{O}_2 \tens \K$), the resulting algebra does contain a projection.
This was proven by Bruce Blackadar and Joachim Cuntz in \cite{BlackadarCuntz:simple}.
This projection, in the nonunital case, plays the role that the unit plays in the unital case.

We begin with a simple characterization of when $\Prim(A)$ is compact.
(Once again, we refer the reader to \ccite{Chapter 4}{Pedersen:CstarBook} for basic theory regarding the primitive ideal space.
An elaboration of this characterization will appear in \cite{TikuisisToms:InProgress}.

\begin{prop}\alabel{CompactPrimCharacterization}
Let $A$ be a $C^*$-algebra.
The following are equivalent.
\begin{enumerate}
\item $\Prim(A)$ is compact;
\item $A$ contains a full element and for every full element $a\in A$, there exists $\e>0$ such that $\|\pi(a)\| \geq \e$ for every nonzero representation $\pi$.
\end{enumerate}
\end{prop}

\begin{proof}
Suppose that $\Prim(A)$ is compact.
Letting $(e_\alpha)$ be an approximate identity, then since the ideal generated by $\{e_\alpha\}$ is all of $A$, we must have by compactness of $\Prim(A)$ that $A$ is generated, as an ideal, by some finite number of $e_\alpha$'s.
In particular, $A$ contains a full element.

Now, letting $a \in A$ be full, note that
\[ I \mapsto \|\pi_I(a)\| \]
is a lower semicontinuous function from $\Prim(A)$ to $(0,\infty)$, and therefore it attains a minimum, $\e > 0$.

Conversely, suppose that $A$ contains an element $a$ such that $\|\pi(a)\| \geq \e$ for every representation $\pi$.
Let $\{U_\alpha\}$ be an open cover of $\Prim(A)$ and let $I_\alpha$ be the ideal of $A$ associated to $U_\alpha$, for each $\alpha$.
Then since $\bigcup U_\alpha = \Prim(A)$, we must have $\overline{\sum I_\alpha} = A$, and in particular, $a$ is approximated by a finite sum; that is, there exists indices $\alpha_1,\dots,\alpha_n$ and some element $b \in \sum_{i=1}^n I_{\alpha_i}$ such that
\[ \|a-b\| < \e. \]
Consequently, we see that for every nonzero representation $\pi$ of $A$,
\[ \|\pi(b)\| \geq \|\pi(a)\|-\e > 0 \]
for every representation $\pi$ of $A$, and therefore, $b$ is full in $A$.
But this means that $\{U_{\alpha_1},\dots,U_{\alpha_n}\}$ is a finite subcover.
Thus, we have proven that $\Prim(A)$ is compact.
\end{proof}

\begin{prop}
\alabel{ASH-SimpleStructure}
Let $A$ be a separable simple approximately subhomogeneous algebra.
Then there exists an inductive limit decomposition
\[ A_1 \labelledrightarrow{\phi_1^2} A_2 \labelledrightarrow{\phi_2^3} \cdots \rightarrow A = \varinjlim A_i \]
such that each $A_i$ is a recursive subhomogeneous algebra with compact spectrum and each connecting map is full (i.e.\ $\phi_i^j(A_i)$ generates $A_j$ as an ideal).
\end{prop}

\begin{proof}
By Corollary \ref{ASH-Structure-Nonunital}, let
\[ B_1 \labelledrightarrow{\phi_1^2} B_2 \labelledrightarrow{\phi_2^3} \cdots \]
be any inductive system of recursive subhomogeneous algebras with $A$ as its limit.
By \ccite{Corollary 5.2}{BlackadarCuntz:simple}, $A \tens \mathcal{O}_2 \tens \K$ contains a nonzero projection.

Hence, $B_i \tens \mathcal{O}_2 \tens \K$ contains a nonzero projection $p$ for some $i$, and without loss of generality we take $i=1$.
The ideal of $B_i \tens \mathcal{O}_2 \tens \K$ generated by $(\phi_1^i \tens \id_{\mathcal{O}_2 \tens \K})(p)$ is of the form
\[ A_i \tens \mathcal{O}_2 \tens \K \]
for some ideal $A_i$ of $B_i$, and $\Prim(A_i) \iso \Prim(A_i \tens \mathcal{O}_2 \tens \K)$ (this follows from \ccite{Corollary 9.4.6}{BrownOzawa:Book}, for instance).
However, since $A_i \tens \mathcal{O}_2 \tens \K$ contains a full projection, we see by Proposition \ref{CompactPrimCharacterization} that its primitive ideal space is compact.
It is evident, by the construction, that $\phi_i^j(A_i)$ generates $A_j$ as an ideal, and since $\overline{\bigcup \phi_i^\infty(A_i)}$ is a nonzero ideal of $A$, it is all of $A$.
Finally, since $A_i$ is an ideal of a recursive subhomogeneous algebra, by Proposition \ref{RSH-ideal}, $A_i$ is recursive subhomogeneous.
\end{proof}

Here is a chief advantage to having an inductive system as in Proposition \ref{ASH-SimpleStructure}.

\begin{prop} (cf.\ \ccite{Proposition 3.2.4}{MyThesis})
\alabel{CompactPrimSimpleLimit}
Let
\[ A_1 \labelledrightarrow{\phi_1^2} A_2 \labelledrightarrow{\phi_2^3} \cdots \]
be an inductive system such that each map $\phi_i^j$ is full and injective, each algebra $A_i$ has compact primitive ideal space, and its limit, $A$, is simple.
Then for every $i$ and every nonzero $a \in A_i$, there exists $j \geq i$ such that $\phi_i^j(a)$ generates $A_j$ as an ideal.
\end{prop}

\begin{proof}
By Proposition \ref{CompactPrimCharacterization}, let $b \in A_i$ be such that $\|\pi(b)\| \geq 1$ for all representations $\pi$.
Since $A$ is simple, there exists $j \geq i$ and some element $c$ in the ideal of $A_j$ generated by $\phi_i^j(a)$ such that
\[ \|\phi_i^j(b)-c\| < 1. \]
For every nonzero representation $\pi$ of $A_j$, since $\phi_i^j$ is full, $\pi \circ \phi_i^j$ is nonzero.
Thus,
\[ \|\pi(c)\| \geq \|\pi(\phi_i^j(b))\| - \|\pi(\phi_i^j(b)-c)\| > 1-1=0. \]
Consequently, we see that $c$ is full, which means that $\phi_i^j(a)$ generates $A_j$ as an ideal.
\end{proof}

\section{Perforation in simple stably projectionless approximately subhomogeneous algebras} 
\alabel{Perforation-section}

The main result here is the existence of a simple, stably projectionless approximately subhomogeneous algebra whose Cuntz semigroup is not almost unperforated.

\begin{thm}\alabel{MainThm}
There exists a simple separable stably projectionless approximately subhomogeneous algebra $A$
such that for any $n \in \mathbb{N}$, there exists $[a],[b] \in \Cu(C)$ and $k \in \mathbb{N}$ such that
\[ (k+1)[a] \leq k[b] \]
yet $[a] \not\leq n[b]$.
In particular, $A$ is not $\jsZ$-stable and $A$ has infinite nuclear dimension.
\end{thm}

The final statements follow from \ccite{Theorem 4.5}{Rordam:Z} (applied to $A \tens \mathcal{K}$) and \cite{Robert:dimNucComp} respectively.
Otherwise, the proof consists of an explicit construction, which will be completed in Section \ref{ProjlessPerfConstruction}.

\subsection{Diagonal maps and perforation}

The inductive limits used to construct our examples will involve particularly tractable homomorphisms, ones which are essentially diagonal.
Observe that recursive subhomogeneous algebras occur as subalgebras of algebras of the form $C(X,M_n)$, and between algebras of this form, we may define bone fide diagonal homomorphisms.
A \textbf{diagonal} map $C(X,M_n) \to C(Y,M_m)$ is a $*$-homomorphism of the form
\begin{align*}
f &\mapsto D_{\alpha_1,\dots,\alpha_p}(f) \\
& := \diag(f \circ \alpha_1, f \circ \alpha_2, \dots, f \circ \alpha_p) \\
&:= \left( \begin{array}{cccc}
f \circ \alpha_1 & 0 & \cdots & 0 \\
0 & f \circ \alpha_2 & \ddots & \vdots \\
\vdots & \ddots & \ddots & 0 \\
0 & \cdots & f \circ \alpha_p & 0
\end{array}\right),
\end{align*}
where $\alpha_1,\dots,\alpha_p:Y \to X$ are continuous functions, called the \textbf{eigenmaps} of $D_{\alpha_1,\dots,\alpha_p}$.

We give now a general criterion for perforation in the Cuntz semigroup of an approximately subhomogeneous algebra.
The key ingredient in the following proof of perforation is a Chern class argument, used initially by Jesper Villadsen in \cite{Villadsen:Perforation}.
A second, but nonetheless crucial, ingredient is the adaptation of Villadsen's construction to positive elements in place of projections, originally pioneered by Andrew Toms in \cite{Toms:annals}.
A comprehensive account of the role of dimension growth in these arguments was given in \cite{TomsWinter:V1}, and the following result (and its proof) is reminiscent of \ccite{Lemma 4.1}{TomsWinter:V1} (the values $d_{i+1}\cdots d_j, p_{i+1}\cdots p_j$ used here play the roles of $N_{i,j},M_{i,j}$ respectively there).

\begin{prop}\alabel{DiagonalPerforation}
Let
\[ A_1 \labelledrightarrow{\phi_1^2} A_2 \labelledrightarrow{\phi_2^3} \cdots \]
be an inductive limit, such that for each $i$, the algebra $A_i$ is a subalgebra of $C(X_i,M_{m_i})$ and $\phi_i^{i+1} = Ad(u) \circ D_{\alpha_1^{(i)},\dots,\alpha_{p_i}^{(i)}}$ for some unitary $u \in C(X_{i+1},M_{m_{i+1}})$ (so that $m_{i+1}=m_ip_i$).
Suppose that $X_i$ contains a copy $Y_i$ of $[0,1]^{d_1\cdots d_{i-1}}$ such that
\begin{itemize}
\item $A_i|_{Y_i} = C(Y_i,M_{m_i})$,
\item For $t=1,\dots,d_i$, $\alpha_t^{(i)}|_{Y_{i+1}}$ is given by the $i^\text{th}$ coordinate projection $([0,1]^{d_1 \cdots d_{i-1}})^{d_i} \to [0,1]^{d_1 \cdots d_{i-1}}$, and
\item For $t=d_i+1,\dots,p_i$, $\alpha_t^{(i)}|_{Y_{i+1}}:Y_{i+1} \to X_i$ factors through the interval.
\end{itemize}
If
\[ \prod_{i=1}^\infty \frac{d_{i+1}}{p_i} > 0 \]
and $p_i > 1$ for all $i$ then for any $n \in \mathbb{N}$, there exists $[a],[b] \in \Cu(\varinjlim A_i)$ and $k \in \mathbb{N}$ such that
\[ (k+1)[a] \leq k[b] \]
yet $[a] \not\leq n[b]$.
\end{prop}

\begin{proof}
Since $\prod_{j=2}^\infty d_j/p_j > 0$ yet $d_j/p_j \leq 1$ for each $j$, we may find an $i$ such that
\[ \prod_{j=i}^\infty d_j/p_j > \frac{6n-1}{6n}, \]
and such that $\dim Y_i \geq 6n$.
We may find subsets $S \subseteq T \subseteq Y_i$ such that
\[ S \iso (S^2)^{3n}, \]
$T$ is open, there exists a retract $r:T \to S$.
(For example, take $S$ to be an embedding of $(S^2)^{3n}$ in the interior of $Y_i$ and let $T$ be a small tubular neighbourhood of $S$.)

Let $\beta \in C(S^2,M_2)$ be a Bott projection, let $f \in C_0(T)$ be a strictly positive function, and define
\[ b_i := f \bigdsum_3 (\beta^{\tens 3n}) \circ r \in C_0(T,\K) \subseteq (A_i \tens \K)_+. \]
Also define
\[ a_i := f 1_2 \in C_0(T,\K)_+ \subseteq (A_i \tens \K)_+. \]
Set
\[ a := \phi_i^\infty(a_i), b := \phi_i^\infty(b_i) \in (\varinjlim A_i \tens \K)_+. \]

If $k \geq 3n+2$ then
\begin{align*}
\rank \bigdsum_k \bigdsum_3 \beta^{\tens 3n} - \rank \bigdsum_{k+1} 1_2 &= (3k-(2k+2)) \\
&= k-2 \\
&\geq 3n \\
&> \frac{\dim (S^2)^{3n} - 1}2
\end{align*}
and therefore by \ccite{Theorem 9.1.2}{Husemoller}
\[ (k+1)[1_2] \leq k\left[\bigdsum_3 \beta^{\tens 3n}\right] \]
in $V(C((S_2)^{3n}))$.
Consequently, we clearly have
\[ (k+1)[a_i] \leq k[b_i] \]
in $\Cu(A_i)$.

To see that $[a] \not\leq n[b]$, suppose the contrary.
Since $f$ is strictly positive on the compact set $S$, let $\e > 0$ be strictly less than its minimum value on that set.
By \textbf{L2}, there exists $j \geq i$ such that
\[ [\phi_i^j((a_i-\e)_+)] \leq n[\phi_i^j(b_i)] \]
in $\Cu(A_j)$.

Set $d := d_i\cdots d_{j-1}$ and $p:=p_i\cdots p_{j-1}$.
We have
\[ X_j = X_i^d, \]
and $\phi_i^j$ includes $d$ coordinate projections.
Set
\[ a' = \phi_i^j((a_i-\e)_+)|_{S^d}, b' = \phi_i^j(b_i)|_{S^d} \in C(S^d,\K) \iso C((S^2)^{3nd},\K). \]
Since Cuntz equivalence passes to quotients, we have
\[ [a'] \leq n[b'] \]
in $\Cu(C((S^2)^{3nd}))$.
With $a'' = 1_{2d} \in C((S^2)^{3nd},\K)$; since $\phi_i^j$ contains $d$ coordinate projections, we have
\[ [a''] \leq [a']. \]
If we now label the $t^\text{th}$ coordinate projection $\delta_t:((S^2)^{3n})^d \to (S^2)^{3n}$, for $t=1,\dots,d$, and set
\[ b'' = \bigdsum_{t=1}^d (\bigdsum_{3n} \beta^{\tens 3n}) \circ \delta_t \dsum 1_{3n(p-d)} \]
then (since there are $3(p-d)$ eigenmaps that aren't coordinate projections, each of which factor through the interval, and Cuntz comparison for functions on the interval is determined by the ranks at each point),
\[ n[b'] \leq [b'']. \]
We therefore have
\[ [a''] \leq [b''] \]
in $V(C((S^2)^{3nd}))$, so that there must exist a projection $c \in C((S^2)^{3nd},\K)$ such that
\[ [a''] + [c] = [b'']. \]
We can compute the rank of $c$ to be $3np-2d$.

We now wish to take the Chern classes to obtain a contradiction.
See \ccite{Section 4.1}{TomsWinter:V1} for background.
The codomain of the Chern class $c$ is the integral cohomology ring
\begin{align*}
&H^*(((S^2)^{3n})^d) = \\
&\quad \mathbb{Z}[e_{s,t}: s=1,\dots,3n, t=1,\dots,d]/\langle e_{s,t}^2: s=1,\dots,3n, t=1,\dots,d\rangle,
\end{align*}
with $c(\beta^{\tens 3n} \circ \delta_t) = (1+e_{1,t}+\cdots+e_{3n,t})$ and therefore,
\[ c(b'') = \prod_{t=1}^3n (1+e_{1,t}+\cdots+e_{3n,t})^{3n}. \]
The coefficient of $\prod_{s,t} e_{s,t}$ is $1 \neq 0$, and since
\[ c(a'')c(c) = c(b'') \]
yet $c(a'')=1$, we must have that the degree of $c(c)$ is at least $3nd$, which imples that $\rank c \geq 3nd$.
That is,
\[ 3nd \leq 3np-2d = (3n-2)p + 2(p-d) \leq (3n-2)p; \]
and by dividing by $3np$, this gives
\[ d/p \leq (3n-2)/3n. \]
However, by assumption, $(6n-1)/6n \leq d/p$ and so
\[ (6n-1)/6n \leq d/p \leq (3n-2)/3n, \]
a contradiction.
\end{proof}

\subsection{Generalizing the Razak building blocks}

Razak introduced certain stably projectionless, though highly nonsimple, $C^*$-algebras in \cite{Razak:classification}.
These building block $C^*$-algebras have trivial $K_0$- and $K_1$-groups and are topologically one-dimensional (in fact, they are fields of $C^*$-algebras over $\mathbb{T}$, which are nontrivial only at the endpoints, $0$ and $1$).
Razak showed that, as one might expect, the simple inductive limits of such building blocks are classified by the traces.
Tsang gave in \cite{Tsang:range} an Effros-Handelman-Shen-type calculation of the range of the invariant for this class of simple $C^*$-algebras, showing in particular that any Choquet simplex can arise as the base of the cone of traces in such an inductive limit.

Our present construction of interesting simple stably projectionless $C^*$-algebras will begin by adapting the building blocks of Razak to allow a base space of high dimension.
The generalized Razak building block construction will centre around what will here be referred to as double-pointed spaces; a double-pointed space is a compact Hausdorff space $X$ together with two (distinct) distinguished points $x_0,x_1 \in X$.
We form a category of double-pointed spaces by imposing that morphisms $(X,x_0,x_1) \to (Y,y_0,y_1)$ are continuous functions $X \to Y$ which send $x_0,x_1$ to $y_0,y_1$ respectively.
We will often use the notation $\X=(X,x_0,x_1)$ to denote a double-pointed space.

Given a double-pointed space $(X,x_0,x_1)$ and a natural number $k$, let us define the $C^*$-algebra
\begin{align*}
R(X,x_0,x_1,k) := \{f \in C(X,M_{k+1}): \exists \lambda \in \C \text{ s.t.\ } f(x_0) &= (\lambda 1_k) \dsum\, 0 \text{ and }\\
 f(x_1) &= \lambda 1_{k+1}\}.
\end{align*}
Razak's original building blocks arise from taking $X=[0,1],x_0=0,x_1=1$.

A small amount of computation characterizes when, up to a unitary conjugation, the image of $R(\X,k)$ under a diagonal map lands in $R(\Y,\ell) \tens M_m$.
The result follows.

\begin{prop}\alabel{CharacterizingExistence-Razak}
Let $\X = (X,x_0,x_1), \Y = (Y,y_0,y_1)$ be double-pointed spaces and let $k,\ell$ be natural numbers.
Let $\alpha_1,\dots,\alpha_p:Y \to X$ be continuous maps.
Then the following are equivalent:
\begin{enumerate}
\item There exists a unitary $u \in C(Y,M_{\ell+1}) \tens M_m$ such that 
\[ uD_{\alpha_1,\dots,\alpha_p}(R(\X,k))u^* \subseteq R(\Y,\ell) \tens M_m; \text{ and} \]
\item Counting multiplicity we have
\begin{align*}
\{\alpha_1(y_0),\dots,\alpha_p(y_0)\} &= a_0\{x_0\} \cup a_1\{x_1\} \cup \ell\{z_1\} \cup \cdots \cup \ell\{z_s\} \text{ and} \\
\{\alpha_1(y_1),\dots,\alpha_p(y_1)\} &= b_0\{x_0\} \cup b_1\{x_1\} \cup (\ell+1)\{z_1\} \cup \cdots \cup (\ell+1)\{z_s\}
\end{align*}
for some points $z_1,\dots,z_s \in X$, and some natural numbers $a_0,a_1,b_0,b_1$ satisfying
\begin{align*}
ka_0 + (k+1)a_1 &= (m-s(k+1))\ell, \text{ and} \\
kb_0 + (k+1)b_1 &= (m-s(k+1))(\ell+1).
\end{align*}
\end{enumerate}
\end{prop}

\begin{proof}
(i) $\Rightarrow$ (ii): If $uR(\X,k)u^* \subseteq R(\Y,\ell) \tens M_m$ then, the unitaries $u_0 = u(y_0)$ and $u_1 = u(y_1)$ must satisfy, for all $f \in R(\X,k)$,
\begin{align*}
u_0 \diag(f(\alpha_1(y_0)),\dots,f(\alpha_p(y_0))) u_0^* &= (1_\ell \dsum 0)  \tens \rho(f) \text{ and} \\
u_1 \diag(f(\alpha_1(y_1)),\dots,f(\alpha_p(y_1))) u_1^* &= 1_{\ell+1} \tens \rho(f),
\end{align*}
for some representation $\rho:R(\X,k) \to M_m$.
Up to unitary equivalence, $\rho$ is the direct sum of irreducible representations,
\[ (\bigdsum_t \psi) \dsum \ev_{z_1} \dsum \cdots \dsum \ev_{z_p} \]
where $z_1,\dots,z_p \in X \setminus \{x_0,x_1\}$ and $\psi:R(\X,k) \to \C$ is the representation given by $\psi(f) = \lambda$ where $f(x_0) = \lambda 1_k \dsum 0$.
Note that $t=m-s(k+1)$.

Noting that $\ev_{\alpha_1(y_0)} \dsum \cdots \dsum \ev_{\alpha_p(y_0)}$ is unitarily equivalent to
\[ \bigdsum_{t\ell} \psi \dsum \bigdsum_\ell \ev_{z_1} \dsum \cdots \dsum \bigdsum_\ell \ev_{z_p} \]
Therefore we see that if $a_0,a_1$ denote the respective multiplicities of $x_0,x_1$ in $\{\alpha_1(y_0),\dots,\alpha_p(y_0)\}$ then
\[ ka_0+(k+1)a_1 = t\ell = (m-s(k+1))\ell. \]
Likewise, we must have
\[ kb_0 + (k+1)b_1 = (m-s(k+1))(\ell+1). \]

Finally, every $z \in X \setminus \{x_0,x_1\}$ occurs an equal number of times in $\{\alpha_1(y_0),\dots,\alpha_p(y_0)\}$ as in $\ell\{z_1\} \cup \cdots \cup \ell\{z_s\}$.
If $q\ell$ denotes this number then $q(\ell+1)$ is the number of times it occurs in $(\ell+1)\{z_1\} \cup \cdots \cup (\ell+1)\{z_s\}$, which is necessarily the same number of times that it occurs in $\{\alpha_1(y_1),\dots,\alpha_p(y_1)\}$.
Hence, we have shown that (ii) holds.

(ii) $\Rightarrow$ (i):
Given (ii), it follows that there exist unitaries $u_0,u_1 \in M_{\ell+1} \tens M_m$ such that, for every $f \in R(X,x_0,x_1,k)$, 
\begin{align*}
u_0 \diag(f(\alpha_1(y_0)),\dots,f(\alpha_p(y_0))) u_0^* &= (1_\ell \dsum 0)  \tens (\lambda 1_t \dsum f(z_1) \dsum \cdots \dsum f(z_s)) \text{ and} \\
u_1 \diag(f(\alpha_1(y_1)),\dots,f(\alpha_p(y_1))) u_1^* &= 1_{\ell+1} \tens (\lambda 1_t \dsum f(z_1) \dsum \cdots \dsum f(z_s)),
\end{align*}
where $t=m-s(k+1)$.

Since the unitary group of $M_{\ell+1} \tens M_m$ is path connected, we can extend $u_0,u_1$ to a homotopy of unitaries $t \mapsto u_t$.
Next, by Urysohn's Lemma, we may find $\eta:Y \to [0,1]$ such that $\eta(y_0) = 0$ and $\eta(y_1) = 1$.
Define $u \in C(Y,M_{\ell+1}) \tens M_m$ by setting
\[ u(y) = u_{\eta(y)}. \]
Evidently,
\[ uD_{(\alpha_1,\dots,\alpha_p)}(R(X,x_0,x_1,k))u^* \subseteq R(Y,y_0,y_1,\ell) \tens M_m. \]
\end{proof}

\begin{remark}\alabel{RazakSoln}
In the sequel, we will use the following explicit solution for (ii), depending on $s,k$ and an additional variable $u$:
\begin{align*}
\ell &:= k+1+2u, \\
m &:= (k+1)^2s, \\
a_0 &:= (k+1)(k+1+u)s, \\
a_1 &:= ksu, \\
b_0 &:= (k+1)su, \\
b_1 &:= k(k+2 + u)s.
\end{align*}
We note that there is a total of $(k^2 + 2ku + 3k + 3u + 2)s$ eigenmaps.
If $\Y = (X^d, (x_0,\dots,x_0, (y_0,\dots,y_0))$, $d \geq b_1$ then this solution permits up to $b_1$ coordinate projections.
Suppose in addition that $\X$ is isomorphic to $(X,x_1,x_0)$ (where we have flipped the points); in this case, we shall use the term \textbf{flipped coordinate projection} to refer to a coordinate projection composed with a homeomorphism which switches $x_0$ and $x_1$.
In addition to the $b_1$ (regular) coordinate projections, we may include up to $a_1$ flipped coordinate projections.
\end{remark}

\subsection{Simplicity}
\alabel{SimpSection}
Suppose we have an inductive system
\[ R(\X_1,k_1) \tens M_{m_1} \labelledrightarrow{\phi_1^2} R(\X_2,k_2) \tens M_{m_2} \labelledrightarrow{\phi_2^3} \cdots \]
where for each $i$, $\phi_i^j$ is the unitary conjugation of a diagonal map, where the eigenmaps include point evaluations at $z^{(i)}_1,\dots,z^{(i)}_{s_{i,j}}$.
(Note that if some point evaluation occurs as an eigenmap of $\phi_i^j$ then it also occurs as an eigenmap of $\phi_i^k$ for $k > j$.)

If, for each $i$, $s_{i,j} \to \infty$ as $j \to \infty$ and the set
\[ \{z^{(i)}_1,z^{(i)}_2,\dots\} \]
is dense in $X_i$ then the inductive limit is simple.
This is easy to verify by noting that, under these conditions, for any nonzero $f \in R(\X_i,k_i) \tens M_{m_i}$, there exists $j$ such that
\[ \phi_i^j(f)(x) \neq 0\ \forall x \in X_j. \]

We can satisfy this simplicity criterion easily in an inductive context with separable spaces $X_i$.
For example, after constructing an initial segment
\[ R(\X_1,k_1) \tens M_{m_1} \to \cdots \to R(\X_i,k_k) \tens M_{m_i}, \]
we can simply pick a dense sequence
\[ \{z^{(i)}_1,z^{(i)}_2,\dots\} \]
and require that for each $t$, point evaluation at $z^{(i)}_t$ occurs as an eigenmap at \textit{some} later stage.

\subsection{The construction}\alabel{ProjlessPerfConstruction}
Our construction consists of an inductive limit $A$ of algebras $A_i = R(\X_i,k_i) \tens M_{m_i}$, 
where
\[ \X_0 = ([0,1],0,1) \]
and
\[ (X_{i+1},x_0^{(i+1)},x_1^{(i+1)}) = \X_{i+1} = \X_i^{d_i} := (X_i^{d_i},(x_0^{(i)},\dots,x_0^{(i)}),(x_1^{(i)},\dots,x_1^{(i)}) ). \]
The map $\phi_i^{i+1}:A_i \to A_{i+1}$ of the inductive system is a unitary conjugations of a diagonal map with $p_i$ eigenmaps, consisting of:
\begin{enumerate}
\item $d_i$ distinct (flip-)coordinate projections $X_i^{d_i} \to X_i$,
\item some number, $s_{i+1}$, of constant maps, and
\item each of the remaining maps factors through the interval $[0,1]$.
\end{enumerate}

The constant maps are included in order to make the limit simple (as explained in Section \ref{SimpSection}).
The remaining maps (in (iii)) are needed in order to satisfy boundary conditions (Theorem \ref{CharacterizingExistence-Razak}),
which ensure that a unitary exists which conjugates the image of $A_i$ into $A_{i+1}$.
We see from Proposition \ref{DiagonalPerforation}
that $\Cu(A)$ is not almost unperforated if
\[ \prod_{i=1}^\infty \frac{d_{i}}{p_i} \neq 0. \]
In fact, Corollary \ref{SDGunperf} will show that if $\prod \frac{d_{i}}{p_i} =0$ then $\Cu(A)$ is almost unperforated.
That is to say, perforation in the Cuntz semigroup of the limit is exactly contingent on whether the system $(A_i,\phi_i^{i+1})$ has slow dimension growth.
We shall see that we can get a system without slow dimension growth, and therefore a limit whose Cuntz semigroup is not almost unperforated.

Naturally, we shall choose the sequence of spaces $A_i=R(\X_i,k_i) \tens M_{m_i}$ inductively, and use Remark \ref{RazakSoln} to aid us in finding the maps $\phi_i^{i+1}$.
Given our choice of algebra $A_i = R(\X_i,k_i) \tens M_{m_i}$, letting $s_i$ denote the required number of constant eigenmaps (as explained in Section \ref{SimpSection}), and letting $u_i$ be specified later, we set
\begin{align*}
k_{i+1} &:= k_i+1+2u_i,  \\
m_{i+1} &:= m_i(k_i+1)^2s_i.
\end{align*}
In $\phi_i^{i+1}$, we can include (up to) $k_i(k_i+2 + u_i)s_i$ coordinate projections and $k_is_iu_i$ flipped coordinate projections, so we set
\[ d_i = k_i(k_i+2 + u_i)s_i + k_is_iu_i = k_i(k_i+2+2u_i)s_i. \]
We have specified the coordinate projection eigenmaps and the constant eigenmaps.
From Theorem \ref{CharacterizingExistence-Razak}, we see that in order to have a unitary that conjugates the image of the diagonal map into $R(\X_{i+1},k_{i+1}) \tens M_{m_{i+1}}$, we need the remaining eigenmaps to take certain values at $x_0^{(i+1)}$ and at $x_1^{(i+1)}$.
By pairing up these values, and finding a path from one to the other, we see that we can find eigenmaps which factor through the interval and which satisfy this boundary behaviour.

Finally, the total number of eigenmaps is $p_i = (2k_iu_i + 3u_i + k_i + 1)s_i$.
Thus, 
\begin{align}
\alabel{RGrowthRatio}
\prod \frac{d_i}{p_i} &= \prod \frac{k_{i}(k_{i}+2+2u_{i})s_{i}}{(k_i^2 + 2k_iu_i + 3k_i + 3u_i + 2)s_i} \\
\notag &= \prod \frac{k_i(k_i+2+2u_i)}{k_i^2 + 2k_iu_i + 3k_i + 3u_i + 2}.
\end{align}

If we set $u_i=k_i$ for each $i$ then we find
\[ k_{i+1} = 3k_i + 1 = 3^{i-1}(k_1 + 1/2) - 1/2 \]
(after solving this recurrence).
For simplicity, set $K=1/3(k_1+1/2)$.
Then \eqref{RGrowthRatio} becomes
\begin{align*}
\prod \frac{d_i}{p_i} &= \prod \frac{(3^{i}K - 1/2)(3^{i}K - 1/2 + 2 + 2\cdot 3^{i}K - 1)}{(3^iK-1/2)^2 + 2(3^iK-1/2)^2 + 3(3^iK - 1/2) + 3(3^iK -1/2) + 2} \\
&= \prod \frac{3\cdot 9^iK^2 - 3^iK - 1/4)} {3\cdot 9^iK^2 + 5/4}
\end{align*}
It is not hard to see that this infinite product is nonzero.
(For instance, take the logarithm and argue that the absolute value of the $i^\text{th}$ term in the log series is eventually smaller than $1/i^2$, which are the terms of a convergent series.)

This concludes the demonstration that there exists a counterexample as named in Theorem \ref{MainThm}, and therefore proves that theorem.

\section{Slow dimension growth and unperforation}
\alabel{SDG-section}

Slow dimension growth is a regularity condition for approximately subhomogeneous algebras that says roughly that the topological dimension is small relative to the matricial dimension.
For simple unital approximately homogeneous algebras (a class that includes all simple approximately homogeneous algebras, up to stable isomorphism), it was introduced by Blackadar, Dadarlat, and R{\o}rdam in \cite{BlackadarDadarlatRordam}, and the definition was adapted to simple unital approximately subhomogeneous algebras by Phillips in \cite{Phillips:arsh}.
Its significance was confirmed by \ccite{Corollary 6.5}{Winter:pure}/\ccite{Corollary 1.3}{Toms:rigidity}, which shows that it is equivalent to $\jsZ$-stability.
For simple separable approximately homogeneous algebras, it has been shown to be equivalent to no dimension growth.
Classification conjectures would predict this also for approximately subhomogeneous algebras (see the range of invariant result in \cite{Elliott:ashrange}).

Any subhomogeneous $C^*$-algebra embeds into continuous functions from the Cantor set into $\K$, and this demonstrates that we should be a bit picky about what space to use to measure the topological dimension.
We find that a good measure of topological dimension comes from looking at the primitive spectrum of a subhomogeneous algebra, and, more specifically, the pieces of the primitive spectrum corresponding to irreducible representations of constant dimension.

\begin{defn}
Let $A$ be a subhomogenous algebra.
For an irreducible representation $\pi$ of $A$, set $d_{top}(\pi) := \dim \Prim_n(A)$ where $n$ is the dimension of the representation $\pi$ (i.e.\ $n$ is such that $\ker \pi \in \Prim_n(A)$).
For an element $a \in A$, the \textbf{dimension-rank ratio} of $a$ is defined to be
\[ \rdratio(a) := \sup \frac{d_{top}(\pi)}{\rank \pi(a)}, \]
where $\pi$ ranges over all the irreducible representations of $A$.
In case $\pi(a)$ vanishes for some irreducible representation $\pi$, we set $\rdratio(a) := \infty$.
\end{defn}

Note that, as long as $a$ is full, $\rdratio(a) < \infty$.

\begin{prop}
\alabel{SDG-Equivalence}
Let
\[ A_1 \labelledrightarrow{\phi_1^2} A_2 \labelledrightarrow{\phi_2^3} \cdots \]
be an inductive system of subhomogeneous $C^*$-algebras with compact, finite-dimensional spectra, and let $A$ be the inductive limit.
Suppose that the maps $\phi_i^j$ are full and injective and that $A$ is simple.
Then the following are equivalent.
\begin{enumerate}
\item For every $i$ and every nonzero $a \in A_i$, $\rdratio(\phi_i^j(a)) \to 0$ as $j \to \infty$; 
\item There exists $a \in A_i$ for some $i$ such that $\rdratio(\phi_i^j(a)) \to 0$ as $j \to \infty$.
\end{enumerate}
\end{prop}

\begin{proof}
Of course, (i) $\Rightarrow$ (ii) is obvious.
Conversely, suppose that (ii) holds as stated, and let $b \in A_j$ be nonzero.
By Proposition \ref{CompactPrimSimpleLimit}, there exists $k \geq j$ such that $\phi_j^k(b)$ generates $A_k$ as an ideal.
In particular, $\rank \pi(\phi_j^k(b)) \geq 1$ for every irreducible representation $\pi$ of $A_k$.
Since $A_k$ is subhomogeneous, let $D$ be the maximal dimension of irreducible representations of $A_k$, so that we see that
\[ \rank \pi(\phi_j^k(b)) \geq D/D \geq \rank \pi(\phi_i^k(a))/D, \]
for every irreducible representation (and therefore, every finite-dimensional representation) of $A_k$.
Hence,
\[ \rdratio(\phi_j^\ell(b)) \leq \rdratio(\phi_i^\ell(a))/D \]
for every $\ell \geq k$, and hence the left hand side approaches $0$ as $\ell \to \infty$.
\end{proof}

\begin{defn}
We say that a simple approximately subhomogeneous algebra $A$ has \textbf{slow dimension growth} if it can be written as an inductive limit
\[ A_1 \labelledrightarrow{\phi_1^2} A_2 \labelledrightarrow{\phi_2^3} \cdots \to A = \varinjlim A_i \]
of subhomogeneous algebras with compact, finite-dimensional spectra and full injective maps, such that the system satisfies the equivalent conditions of Proposition \ref{SDG-Equivalence}. 
\end{defn}

The main technical tool used to show that slow dimension growth implies unperforated Cuntz semigroup is the radius of comparison, which we define presently.
The radius of comparison was originally defined for a unital $C^*$-algebra in \cite{Toms:flat}.
It is not hard to see that it is a property of the Cuntz semigroup.
We use the same definition, except that we allow ourselves to normalize functionals against an element that may not be the unit, as is necessary to adapt this concept to nonunital algebras.

\begin{defn}
Let $S$ be an ordered semigroup from the category $\catCu$ such that $x \geq 0$ for all $x \in S$.
Define $F(S)$ to be the set of all functions $\lambda:S \to [0,\infty]$ which are:
\begin{enumerate}
\item linear, meaning that they are additive and send $0$ to $0$,
\item order-preserving, and
\item supremum-preserving, for suprema of increasing sequences.
\end{enumerate}
Let $e \in S$ be full.
The \textbf{radius of comparison} of $S$, with respect to $e$, is the infimum of real numbers $r > 0$ such that
\begin{itemize}
\item[{\bf RC}] \alabel{RCcond}
If $x,y \in S$ satisfy $\lambda(x) + r < \lambda(y)$ for all $\lambda \in F(S)$ for which $\lambda(e)=1$ then $x \leq y$.
\end{itemize}
We denote the radius of comparison of $S$ with respect to $e$ by $\rc(S,e)$.
\end{defn}

A definition of the radius of comparison for nonunital $C^*$-algebras was already given in \cite{BRTTW}, although that definition is slightly different, in that it uses $\leq$ in place of $<$ in {\bf RC}.
(This only makes a difference in cases where there aren't enough functionals taking values other than $0$ and $\infty$, and in particular, \ccite{Proposition 3.2.3}{BRTTW} shows that the definitions agree for stably finite simple $C^*$-algebras.)

We next obtain a bound on the radius of comparison for a recursive subhomogeneous algebra.
This result implies that, if $(R_i,\phi_i^{i+1})$ is an inductive system of recursive subhomogeneous algebras with slow dimension growth, then the radius of comparison of $R_i$ approaches zero as $i \to \infty$.
In the unital case, this bound has been observed in \ccite{Theorem 5.1}{Toms:comparison}.

\begin{prop}\alabel{RCrsh}
Let $R$ be a recursive subhomogeneous algebra with canonical representation $\sigma:R \to C(\Omega,\K)$ and with a full element $e \in R_+$.
Then
\begin{equation}
\alabel{RCrshEq}
\rc(\Cu(R),[e]) \leq \rdratio(e).
\end{equation}
\end{prop}

\begin{proof}
It was stated in \ccite{Corollary 3.4}{CommutativeCuntz} that, if $R$ is unital, $a,b \in (R \tens \K)_+$ and
\[ \rank \sigma(a)(\omega) + \frac{d_{top}(\omega)-1}2 \leq \rank \sigma(b)(\omega) \]
for all $\omega \in \Omega$, then $[a] \leq [b]$ in $\Cu(R)$.
However, the proof of \ccite{Corollary 3.4}{CommutativeCuntz} does not at all use the hypothesis that $R$ is unital (rather, non-unital recursive subhomogeneous algebras were not defined in the literature at the time that \cite{CommutativeCuntz} was written).
Hence, this continues to hold without assuming that $R$ is unital.
From this, \eqref{RCrshEq} is evident.
\end{proof}

The radius of comparison for a unital $C^*$-algebra $A$ is defined to be $\rc(\Cu(A),[1])$, and \ccite{Propositions 3.2.3 and 3.2.4(iii)}{BRTTW} shows that it behaves well with respect to inductive limits (see also \ccite{Proposition 3.3}{Toms:InfiniteFamily}).

The proof of \ccite{Proposition 3.2.4(iii)}{BRTTW} is contingent upon the fact that $[1] \ll [1]$.
In the case of nonunital recursive subhomogeneous algebras, if we let $e$ be a strictly positive element, we may not have $[e] \ll [e]$, though we do have the following.

\begin{prop}
\alabel{WeakUnitRSH}
Let $R$ be a recursive subhomogeneous algebra with finite-dimensional total space and compact spectrum, and let $e \in R_+$ be strictly positive.
Then there exists $\ell$ such that $[e] \ll \ell[e]$ in $\Cu(R)$.
\end{prop}

\begin{proof}
Let us assume that we have a recursive subhomogeneous decomposition for $R$, and let $\Omega$ be its total space.
Set
\[ \Omega_0 := \{\omega \in \Omega : \sigma(R)(\omega) \neq 0\}. \]
By Proposition \ref{RCrsh}, there exists $r$ such that if $a,b \in (R \tens \K)_+$ satisfy
\[ \rank \sigma(a)(\omega) + r \leq \rank \sigma(b)(\omega) \]
for all $\omega \in \Omega_0$, then $[a] \leq [b]$ in $\Cu(A)$.
(This does follow from Proposition \ref{RCrsh}, but in this formulation, it is more easily derived from its proof.)

Since the spectrum of $R$ is compact, let $\e > 0$ be such that $\|\pi(e)\| \geq \e$ for every nonzero representation $\pi$; in particular,
\[ \sigma((e-\e)_+)(\omega) \neq 0 \]
for all $\omega \in \Omega_0$.
Let $D := \max_{\omega \in \Omega} \rank \sigma(e)(\omega)$.
Then we see that, for all $\omega \in \Omega_0$,
\[ \rank \sigma(e)(\omega) + r \leq D+r \leq \rank \sigma(\bigdsum_{D+r} (e-\e)_+) \]
and therefore,
\[ [e] \leq (D+r)[(e-\e)_+] \ll (D+r)[e], \]
as required.
\end{proof}

\begin{remark*}
Suppose that $S \in \catCu$ has a maximal element, $\infty$ (this is the case for the Cuntz semigroup of a $\sigma$-unital algebra, for example).
For $e \in S$ full, the existence of $\ell$ such that $e \ll \ell e$ is equivalent to $e \ll \infty$.

To see this, suppose that $e \ll \infty$.
Since $\infty$ is the supremum of an $\ll$-increasing sequence, there exists $f_1 \ll f_2 \ll \infty$ such that $e \leq f_1$.
Since $e$ is full, $\infty = \sup ne$ so that $f_2 \leq ne$ and therefore, $e \leq f_1 \ll f_2 \leq ne$.
\end{remark*}

\begin{lemma}(cf.\ \ccite{Lemma 3.4}{Toms:InfiniteFamily})
\alabel{UFfunctional}
Let
\[ S_1 \labelledrightarrow{\alpha_1^2} S_2 \labelledrightarrow{\alpha_2^3} \cdots \]
be an inductive system in $\catCu$ with limit $S$,
and let $x \in S, y_i \in S_i$ such that $\alpha_i^{i+1}(y_i) \leq y_{i+1}$ for all $i$, and set $y := \sup \alpha_i^\infty(y_i)$.
If
\[ \lambda(\alpha_1^\infty(x)) < \lambda(y) \]
for all $\lambda \in F(S)$ for which $\lambda(x) \in (0,\infty)$ and if $x_1 \in S_1$ satisfies
\[ \alpha_1^\infty(x_1) \ll x \ll \infty \alpha_1^\infty(x_1) \]
then there exists $j \geq 1$ such that
\[ \lambda(\alpha_1^j(x_1)) < \lambda(y_j) \]
for all $\lambda \in F(S_j)$ for which $\lambda(\alpha_1^j(x_1)) \in (0,\infty)$.
\end{lemma}

\begin{proof}
Suppose, for a contradiction, that for each $i$ there exists $\lambda_i \in F(S_i)$ such that $\lambda_i(\alpha_1^i(x_1)) = 1$ yet
\[ \lambda_i(\alpha_1^i(x_1)) \geq \lambda_i(y_i). \]
Let $\beta$ be an ultrafilter and define $\lambda:\bigcup \alpha_i^\infty(S_i) \to [0,\infty]$ by
\[ \lambda(\alpha_i^\infty(z)) = \lim_\beta \lambda_j(\alpha_i^j(z)), \]
for $z \in S_i$.
Then $\lambda$ is additive, order-preserving, and satisfies
\[ 1=\lambda(\alpha_1^\infty(x_1))\geq \liminf \lambda(\alpha_i^\infty(y_i)), \]
although it may not be lower semicontinuous.

Define $\tilde\lambda:S \to [0,\infty]$ by
\[ \tilde\lambda(z) = \sup \{\lambda(z'): z' \in \bigcup \alpha_i^\infty(S_i), z' \ll z\}. \]
Then by the proof of \ccite{Lemma 4.7}{ElliottRobertSantiago}, $\tilde\lambda \in F(S)$.
Clearly, for $z \in \bigcup \alpha_i^\infty(S_i)$, $\tilde\lambda$ satisfies
\[ \tilde\lambda(z) \leq \lambda(z). \]
We have
\begin{align*}
\tilde\lambda(y) &= \lim \tilde\lambda(\alpha_i^\infty(y_i)) \\
&\leq \liminf \lambda(\alpha_i^\infty(y_i)) \\
&\leq \lambda(\alpha_1^\infty(x_1)) \\
&\leq \tilde\lambda(x),
\end{align*}
yet, since $x \leq nx_1$ for some $n$,
\[ 0 < \lambda(\alpha_1^\infty(x_1)) \leq \tilde\lambda(x) \leq n\tilde\lambda(x_1) < \infty. \]
This is a contradiction to our hypotheses.
\end{proof}

\begin{thm}(cf.\ \ccite{Proposition 3.2.4 (iii)}{BRTTW})
\alabel{RCindlimits}
Let
\[ S_1 \labelledrightarrow{\alpha_1^2} S_2 \labelledrightarrow{\alpha_2^3} \cdots \]
be an inductive system in $\catCu$ with limit $S$, and let $e_i \in S_i, e \in S$ such that $\alpha_i^j(e_i) = e_j$ for all $i \leq j$ and $\alpha_i^\infty(e_i) = e$.
If, for some $k,\ell$, we have $ke_1 \ll \ell e_1$ then
\[ \rc(S,e) \leq \frac{\ell}{k} \liminf \rc(S_i,e_i). \]
\end{thm}

\begin{proof}
By restricting to a subsequence, it suffices to show that, if $\rc(S_i,e_i) < r$ for all $i$ 
then $\rc(S,e) \leq \frac{\ell}k r =: r'$.
By density of the rationals, it in fact suffices to assume that $r=p/q$ for some natural numbers $p,q$.
With this goal in mind, let $x,y \in S$ satisfy
\begin{equation}\alabel{RCindlimits-begineqn} 
\lambda(x) + r' < \lambda(y)
\end{equation}
for all $\lambda \in F(S)$ for which $\lambda(e)=1$.
Equivalently, if $F_0$ denotes the set of all $\lambda$ for which $\lambda(e) \in (0, \infty)$, then by multiplying \eqref{RCindlimits-begineqn} by $qk$,
\[ \lambda(k q x + \ell p e) < \lambda(k q y)\quad \forall \lambda \in F_0. \]
Notice that as a consequence, we have
\[ F_0 = \{\lambda \in F(S): \lambda(k q x + \ell p e) \in (0,\infty)\}. \]

Now, by {\bf L1}, let $x_i,y_i \in S_i$ such that $\alpha_i^{i+1}(x_i) \ll x_{i+1}, \alpha_i^{i+1}(y_i) \leq y_{i+1}$ for all $i$, and $x=\sup \alpha_i^\infty(x_i), y=\sup \alpha_i^\infty(y_i)$.
For each $i$, and each $x' \in S_i$ for which $x' \ll x_i$, we have
\[ \alpha_i^\infty(k q x' + k p e_i) \ll k q x + \ell p e. \]
Therefore, by Lemma \ref{UFfunctional} with $(S_j)_{j \geq i}$ in place of $(S_j)_{j \geq 1}$, $kqx'$ in place of $x_1$, $kqx+\ell pe$ in place of $x$, and $kqy_j$ in place of $y_j$, there exists $j \geq i$ such that
\[ \lambda(\alpha_i^j(k q x') + k p e_j) < \lambda(k q y_j) \]
for all $\lambda \in F(S_j)$ satisfying $\lambda(\alpha_i^j(k q x') + k p e_j) \in (0,\infty)$.
Notice again, since $x' \ll x_i \leq \infty e$, we have $\lambda(\alpha_i^j(k q x') + k p e_j) \in (0,\infty)$ if and only if $\lambda(e_j) \in (0,\infty)$.
That is, we have
\[ \lambda(\alpha_i^j(x')) + r < \lambda(y_j) \]
for all $\lambda \in F(S_j)$ satisfying $\lambda(e_j) = 1$.
Since $\rc(S_j,e_j) < r$, it follows that $\alpha_i^j(x') \leq y_j$.
As $x' \ll x_i$ is arbitrary, $\phi_i^\infty(x_i) \leq y$, and since $i$ is arbitrary, $x \leq y$, as required.
\end{proof}

\begin{cor}\alabel{SDGunperf}
Let $A$ be a simple separable approximately subhomogeneous algebra with slow dimension growth.
Then $\Cu(A)$ is almost unperforated.
\end{cor}

\begin{proof}
By Propositions \ref{SDG-Equivalence} and \ref{RCrsh}, there exists an inductive system
\[ A_1 \labelledrightarrow{\phi_1^2} A_2 \labelledrightarrow{\phi_2^3} \cdots \to A = \varinjlim A_i \]
and full elements $e_i \in (A_i)_+$ such that $\phi_i^j(e_i) = e_j$, $e_i \ll \infty$, and
\[ \lim_{i \to \infty} \rc(\Cu(A_i),[e_i]) = 0. \]
Thus by Theorem \ref{RCindlimits},
\[ \rc(\Cu(A),[e]) = 0, \]
where $e = \phi_i^\infty(e_i)$.

That $\Cu(A)$ is almost unperforated follows now from \ccite{Proposition 3.3.3 (ii)}{BRTTW}.
\end{proof}

\end{document}